\documentclass[12pt,a4paper]{article}
\usepackage{fullpage}
\usepackage[utf8]{inputenc}
\usepackage[english]{babel}
\usepackage{csquotes}
\usepackage{amsmath, amssymb, amsthm}
\usepackage{hyperref}

\hypersetup{
  pdftitle={The coarea formula for projections of Lipschitz mappings on Carnot groups},
  pdfauthor={Sergey Basalaev},
  colorlinks,
  linkcolor=blue,
  citecolor=blue
}

\theoremstyle{plain}
\newtheorem{theorem}{Theorem}
\newtheorem{lemma}[theorem]{Lemma}
\newtheorem{corollary}[theorem]{Corollary}

\theoremstyle{definition}
\newtheorem{definition}[theorem]{Definition}

\numberwithin{theorem}{section}
\numberwithin{equation}{section}

\begin{document}
\selectlanguage{english}

\title{The coarea formula for projections of Lipschitz mappings
on Carnot groups}
\author{Sergey Basalaev\thanks{The work is supported by the
Mathematical Center in Akademgorodok under the agreement
No.\ 075-15-2022-282 with the Ministry of Science and Higher
Education of the Russian Federation.}}
\maketitle

\begin{abstract}
A special type of coarea inequality is proved for compositions of
Lipschitz mappings of Carnot groups with projections along
horizontal vector fields. It is proved that the equality is
achieved for mappings with finite codistortion and mappings on
the Heisenberg group.


\end{abstract}

\section{Introduction}

H.\ Federer in~\cite{FedCoarea} established
the well-known coarea formula
\begin{equation}
\label{EqEuclideanCoarea}
  \intop_{\mathbb{R}^k}
  \mathcal{H}^{n-k}(f^{-1}(y)) \, dy
  =
  \intop_{\Omega} J(x, f) \, dx
\end{equation}
where $f : \Omega \to \mathbb{R}^k$,
$\Omega \subset \mathbb{R}^n$, is a Lipschitz mapping and
$J(x, f) = \sqrt{\det Df(x) Df(x)^*}$.
This formula has numerous applications
in analysis on Euclidean
spaces~\cite{FedBook, EGBook, GMSBook, LiYaBook}.
It extends naturally to the Riemannian spaces~\cite{ChavezBook}
and was also extended to Sobolev mappings of Euclidean
spaces~\cite{MSZ2003} and to Lipschitz mappings defined on
$\mathcal{H}^k$-rectifiable sets in metric spaces~\cite{AK2000}.

Nowadays, many topics of geometric measure theory are studied in
sub-Riemannian spaces which naturally arise e.\,g.\ in the theory
of subelliptic equations, contact geometry, optimal control
theory, non-holonomic mechanics (see \cite{MoBook} for a great
review of the applications).
Our interest is in a sub-Riemannian analogue of a
particular case of this formula. Given a Lipschitz mapping
$\varphi : \Omega \to \mathbb{R}^n$, $\Omega \subset \mathbb{R}^n$, on Euclidean space and an orthogonal projection
$\mathrm{Pr}_j$ on the hyperplane $\{ x_j = 0 \}$ the
formula~\eqref{EqEuclideanCoarea} for
$\mathrm{Pr}_j \circ \varphi$ may be written as
\begin{equation}
\label{EqProjCoarea}
  \intop_{\mathbb{R}^{n-1}}
  \mathcal{H}^1 \bigl( (\mathrm{Pr}_j \circ \varphi)^{-1}(y) \bigr) \, dy
  = \intop_\Omega \left|
  \mathop{\mathrm{adj}} D\varphi(x)
  \langle \partial_{x_j} \rangle \right| \, dx.
\end{equation}
P.\ Pansu in~\cite{Pa1982, Pa1989} introduced
a suitable notion of a differential $D_\mathcal{P}$
for mappings between Carnot groups with Carnot--Carath\'{e}odory
metrics and proved that Lipschitz continuous mappings between
such groups are differentiable a.\,e.\ in this sense.
In this paper we prove
(see the precise statement in Theorem~\ref{ThCoareaIneq})
that for a Lipschitz mapping
$\varphi : \Omega \to \mathbb{G}$, $\Omega \subset \mathbb{G}$,
on a Carnot group $\mathbb{G}$ and a projection $\mathrm{Pr}_j$
on the hyperplane $\Pi_j = \{ x_j = 0 \}$ along the integral lines of a
horizontal left-invariant vector field $X_j$ the following
inequality holds
\begin{equation}
\label{EqIntroIneq}
  \intop_{\mathbb{R}^{n-1}}
  \mathcal{H}^1 \bigl( (\mathrm{Pr}_j \circ \varphi)^{-1}(y) \bigr) \, dy
  \le C \intop_\Omega \left|
  \mathop{\mathrm{adj}} D_\mathcal{P}\varphi(x)
  \langle X_j \rangle \right| \, dx.
\end{equation}
The inequality is still useful for geometric measure theory
since it leads to a Sard type property that the set of singular
values of $\mathrm{Pr}_j \circ \varphi$ is negligible. However,
it is interesting even on its own because, despite
extensive research, the coarea formula for mappings between
Carnot groups is still not well understood.
Most of the results in this direction are established only for
sufficiently smooth contact mappings~\cite{Pa1982, He1994,
Magn2004, Magn2006, Magn2008, KarmVod2009, KarmVod2013, MontiVit2015}
or for Lipschitz mappings in the special case where
$\ker D_\mathcal{P} f$ is a normal
subgroup of the codomain~\cite{Magn2002, Magn2005, Magn2011, Corni2022, JuliaGolo2022}. When the latter condition is not
fullfilled very little is known. Even the case of the mapping
$\mathbb{H}^1 \to \mathbb{R}^2$, $\mathbb{H}^1$ being the
Heisenberg group, is a very challenging problem~\cite{LeoMagn2011, Kozh2011, KozhPhD, Bas2015, Magn2018}.

We emphasize that the inequality~\eqref{EqIntroIneq} is new
and not a consequence of any of these works. The major obstacle
here is that the projection $\mathrm{Pr}_j$ is generally not
contact and therefore not Lipschitz in
Carnot--Carath\'{e}odory metric. As a consequence the
composition $\mathrm{Pr}_j \circ \varphi$ does not have enough
regularity in both Riemannian and sub-\hspace{0pt}Riemannian
sense. Some developments for coarea of non-contact mappings
may be found in~\cite{Karm2019, Karm2024} though only for
$C^1$-smooth mappings with non-vanishing Jacobian. The case
of regular values is actually quite straightforward in our
case since it reduces to a combination of Fubini theorem and
a change of variable. The case of singular values is what
presents a real challenge. The key idea here is that
while projections behave badly from the metric point of view
they have nice properties in terms of measure. This allows us
to prove Eilenberg type inequality~\ref{LemmaCoareaAC}. Then,
we modify and adapt the reasoning of~\cite{Magn2002} to prove the
inequality~\eqref{EqIntroIneq}.

The paper is structured as follows. In Section~2 we introduce
Carnot groups and all the necessary structures on them.
In Section~3 we define projections along the vector fields
and derive their basic properties. As a technical tool a
Fubini type Theorem~\ref{ThFubini} is proved. In section~4
we state and prove the coarea inequality,
Theorem~\ref{ThCoareaIneq}, and some cases where the equality
is achieved, namely functions of finite codistortion
(Corollary~\ref{ThCoareaFinCod}) and
functions on $\mathbb{H}^1$ (Corollary~\ref{ThCoareaH1}).

\section{Carnot groups}

Let $\mathbb{G}$ be a connected simply connected nilpotent Lie
group. The group $\mathbb{G}$ is \emph{graded}
(see e.\,g.~\cite{RS1976, FS, BLU}) if there are fixed subspaces
$V_1, \ldots, V_m$ of its Lie algebra of left-invariant vector
fields $\mathfrak{g}$ such that
\[
  \mathfrak{g} = V_1 \oplus \dots \oplus V_m
  \quad \text{and} \quad
  [V_1, V_k] \subset V_{k+1},
  \quad k = 1, \ldots, m,
\]
where we assume $V_{m+1} = \{ 0 \}$.
Denote $n_i = \dim V_i$,
$N = n_1 + \ldots + n_m = \dim \mathfrak{g}$.
% The tuple $(n_1, \ldots, n_m)$ is the \emph{flag} of the group.
The group $\mathbb{G}$ is \emph{stratified}
or is a \emph{Carnot group} if $\mathfrak{g}$ is generated
by $V_1$, i.\,e. $[V_1, V_k] = V_{k+1}$ for $k = 1, \ldots, m$.

We call a collection of linear independent
left-invariant vector fields
$X_1, \ldots, X_N$ the \emph{graded basis} of
$\mathfrak{g}$ if for each $k = 1, \ldots, m$
some subcollection of these fields is a basis of~$V_k$.
To each $X_i$ we assign a formal degree
$\sigma_i = k$ if $X_i \in V_k$. We assume that $X_i$
are ordered by degree, in particular $X_1, \ldots, X_{n_1}$
span $V_1$.

The exponential mapping
$g = \exp \Big(\sum\limits_{i=1}^N x_i X_i \Big)(e)$
($e$~being the neutral element of $\mathbb{G}$)
is a global diffeomorphism and defines a system of
\emph{canonical coordinates of the 1st kind} on~$\mathbb{G}$.
For convenience we identify the point $g \in \mathbb{G}$ with
its coordinates $(x_1, \ldots, x_N) \in \mathbb{R}^N$,
so diffeomorphisms $\exp : \mathfrak{g} \to \mathbb{G}$ and
$\log : \mathbb{G} \to \mathfrak{g}$ are identities
\[
  \exp \Big(\sum\limits_{i=1}^N x_i X_i \Big)
  = (x_1, \ldots x_N),
  \qquad
  \log (x_1, \ldots x_N) = \sum\limits_{i=1}^N x_i X_i.
\]
The group operation $x \cdot y$ in the canonical coordinates
takes the form
\begin{equation}
\label{EqGroupMultCoords}
(x \cdot y)_i = x_i + y_i
  + \sum_{\sigma_k > \sigma_i} P_{ik}(x, y),
\end{equation}
where $P_{ik}$ are polynomials depending only on coordinates
of degree less than $\sigma_i$.
In particular, $(x \cdot y)_i = x_i + y_i$ for $i = 1, \ldots, n_1$.

The \emph{dilations} $\delta_\lambda$ defined by
$\delta_\lambda (x_1, \ldots, x_N) =
(\lambda^{\sigma_1} x_1, \ldots, \lambda^{\sigma_N} x_N)$
are group automorphisms for all $\lambda > 0$.
We denote
the \emph{homogeneous dimension} of $\mathbb{G}$ as
\[
  \nu = \sum\limits_{k=1}^m k \, n_k
      = \sum\limits_{i=1}^N \sigma_i.
\]

We call the subspace
$V_1 = H \mathbb{G}$ a \emph{horizontal space}
of~$\mathfrak{g}$, its elements are the
\emph{horizontal vector fields}. An absolutely
continuous curve $\gamma$ is \emph{horizontal}
if $\dot\gamma \in H\mathbb{G}$ a.\,e.
If $\mathbb{G}$ is a Carnot group, by Chow--Rashevskii theorem
any two points can be connected by a horizontal curve.
The sub-\hspace{0pt}Riemannian structure on Carnot group
$\mathbb{G}$ is given by a scalar product on $H \mathbb{G}$.
The \emph{Carnot--Carath\'{e}odory} distance
$\rho(x,y)$ between two points $x, y \in \mathbb{G}$
is the infimum of lengths of horizontal curves with the
endpoints $x$ and $y$. It is left-invariant and
$\delta_\lambda$-homogeneous:
\[
  \rho(x, y) = \rho(zx, zy),
  \quad
  \rho(\delta_\lambda x, \delta_\lambda y)
  = \lambda \rho(x, y),
  \qquad
  x, y, z \in \mathbb{G}, \:\: \lambda > 0.
\]
The topology given by the Carnot--Carath\'{e}odory
metric is equivalent to the Euclidean one.
However, unless $H \mathbb{G} = T \mathbb{G}$, the metric
itself is not equivalent to any Riemannian metric.
We denote an open ball in the metric~$d$ as $B_r(x)$,
where $x \in \mathbb{G}$, $r > 0$ is a radius. We may omit
$x$ for the ball centered at the origin,
i.\,e.\ $B_r = B_r(0)$.

For the purposes of proofs we extend the scalar product
to the whole $\mathfrak{g}$ in such a way that the subspaces
$V_1, \ldots, V_m$ are orthogonal. In particular, we may assume
that $X_1, \ldots, X_N$ is an orthonormal basis.
So, $\mathbb{G}$ is also equipped with the left-invariant
Riemannian structure. The metric and measures corresponding the
Riemannian structure are denoted, e.\,g.\ as
$\rho_{\mathrm{riem}}$ or $\mathcal{H}^1_{\mathrm{riem}}$.
Note, that the results obtained are indepenent of a particular
extension of a scalar product and so only depend on a
sub-\hspace{0pt}Riemannian structure of a group.

The Lebesgue measure $\mathcal{L}^N$ on $\mathbb{R}^N$ is a
bi-invariant Haar measure on $\mathbb{G}$. We denote
integration by $\mathcal{L}^N$ simply as $dx$.

\begin{definition}
\emph{Outer $s$-dimensional (spherical) Hausdorff $\varepsilon$-measure}
$\mathcal{H}^s_\varepsilon$, $s \ge 0$,
$\varepsilon > 0$, of $E \subseteq \mathbb{G}$
with respect to Carnot--Carath\'{e}odory distance $\rho$
is defined as
\[
  \mathcal{H}^s_\varepsilon (E) = \beta_s \inf \Bigl\{
    \sum_{k=1}^\infty r_k^s :
    E \subseteq \bigcup_{k \in \mathbb{N}} B(x_k, r_k),
    \:
    \sup_{k \in \mathbb{N}} r_k \le \varepsilon
  \Bigr\},
\]
where $\beta_s > 0$ is some normalization constant.
We do not specify the value of this constant besides
$\beta_1 = \omega_1 = 2$.\footnote{
In Euclidean spaces the value of $\beta_s$ is usually assumed
to be
$\omega_s = \frac{\pi^{s/2}}{\Gamma(1 + s/2)}$.
With such a choice the measure $\mathcal{H}^k$ coinsides with
the Lebesgue measure $\mathcal{L}^k$ on $k$-dimensional
subspaces. On Carnot group it may be reasonable to choose
other values, for instance such that $\mathcal{H}^\nu$
coinsides with $\mathcal{L}^N$.}
\end{definition}

In the subsequent we use the following properties
of $\varepsilon$-measures.

\begin{lemma}[Properties of Hausdorff $\varepsilon$-measures]
\label{EpsHausProps}
~

\begin{enumerate}
\item
$\mathcal{H}^s_\varepsilon(U) < +\infty$
for every bounded $U \subset \mathbb{G}$;

\item
$\mathcal{H}^s_\varepsilon(x \cdot E)
= \mathcal{H}^s_\varepsilon(E)$
for all $E \subseteq \mathbb{G}$, $x \in \mathbb{G}$;

\item
$\mathcal{H}^s_\varepsilon(\delta_r E)
= r^s \mathcal{H}^s_{\varepsilon/r}(E)$
for all $E \subseteq \mathbb{G}$, $r > 0$;

\item
If $K_n \subset \mathbb{G}$, $n \in \mathbb{N}$,
is a decreasing sequence of nested compact sets then
\[
  \mathcal{H}^s_\varepsilon \Big(
  \bigcap\limits_{n \in \mathbb{N}} K_n \Big)
  = \lim\limits_{n \to \infty} \mathcal{H}^s_\varepsilon(K_n).
\]
\end{enumerate}

\end{lemma}

\begin{proof}
The properties 1--3 are fairly obvious. For the last property
let $K = \bigcap_{n \in \mathbb{N}} K_n$. Then
$\mathcal{H}^s_\varepsilon(K) \le \mathcal{H}^s_\varepsilon(K_n)$ for all $n \in \mathbb{N}$. On the other hand, since
$K$ is a compact, for each $\delta > 0$ it can be covered by
a finite family of balls $B_k = B(x_k, r_k)$,
$k = 1, \ldots, m$, such that $r_k \le \varepsilon$ and
\[
  \beta_s \sum_{k=1}^m r_k^s
  < \mathcal{H}^s_\varepsilon(K) + \delta.
\]
Since $\bigcup B_k$ is open there is some $n \in \mathbb{N}$
such that $K_n \subset \bigcup B_k$ and therefore
\[
  \mathcal{H}^s_\varepsilon(K_n)
  \le \beta_s \sum_{k=1}^m r_k^s
  < \mathcal{H}^s_\varepsilon(K) + \delta.
\]
It follows that
$\mathcal{H}^s_\varepsilon(K_n) \to
\mathcal{H}^s_\varepsilon(K)$ as $n \to \infty$
and the lemma is proved.
\end{proof}

\begin{definition}
\emph{Outer $s$-dimensional (spherical) Hausdorff measure}
$\mathcal{H}^s$, $s > 0$,
of $E \subseteq \mathbb{G}$ with respect to $d_{cc}$
is defined as
\[
  \mathcal{H}^s (E) =
  \lim_{\varepsilon \to 0+} \mathcal{H}^s_\varepsilon(E)
  = \sup_{\varepsilon > 0} \mathcal{H}^s_\varepsilon(E).
\]
\end{definition}

Recall that $\mathcal{H}^s$ is a regular Borel measure
on $\mathbb{G}$. The following properties easily follow
from its definition.

\begin{lemma}[properties of Hausdorff measures]
\label{HausProps}
~

\begin{enumerate}

\item
$\mathcal{H}^s(x \cdot E) = \mathcal{H}^s(E)$
for all $E \subseteq \mathbb{G}$, $x \in \mathbb{G}$;

\item
$\mathcal{H}^s(\delta_r E) = r^s \mathcal{H}^s(E)$
for all $E \subseteq \mathbb{G}$, $r > 0$.

\item
$\mathcal{H}^\nu = \Theta_\mathbb{G} \mathcal{L}^N$
for some constant $\Theta_\mathbb{G} > 0$
where $\mathcal{L}^N$ is a Lebesgue measure on $\mathbb{G}$.

\end{enumerate}
\end{lemma}

The latter holds due to the fact that
$\mathcal{H}^\nu$ is a Radon measure on $\mathbb{G}$,
both measures are left-invariant and homogeneous with respect
to $\delta_r$:
\[
  \frac{\partial \mathcal{H}^\nu}{\partial \mathcal{L}^N}(x)
  = \lim_{r \to 0} \frac{\mathcal{H}^\nu(B_r(x))}{\mathcal{L}^N(B_r(x))}
  = \lim_{r \to 0} \frac{r^\nu \mathcal{H}^\nu(B_1)}{r^\nu \mathcal{L}^N(B_1)}
  = \frac{\mathcal{H}^\nu(B_1)}{\mathcal{L}^N(B_1)}
  = \Theta_{\mathbb{G}}.
\]

\section{Hyperplanes and projections}

\begin{definition}
For each $j = 1, \ldots, n_1$ and $s \in \mathbb{R}$
define the \emph{vertical hyperplane}
\[ \Pi_{s,j} = \{ x \in \mathbb{G} : x_j = s \}. \]
Denote $\Pi_j = \Pi_{0,j}$. Define also the axis
$L_j = \{ \exp(sX_j) : s \in \mathbb{R} \}$.
\end{definition}

\begin{lemma}
\label{LemmaSubgroups}
For all $j = 1, \ldots, n_1$
\begin{enumerate}

\item
$L_j$ is an Abelian subgroup of $\mathbb{G}$
of homogeneous degree~$1$.

\item
The sub-Riemannian Hausdorff measure $\mathcal{H}^1$,
Riemannian Hausdorff measure $\mathcal{H}^1_{\mathrm{riem}}$
and Lebesgue measure $\mathcal{L}^1$ coincide on $L_j$.

\item
$\Pi_j$ is a normal subgroup of $\mathbb{G}$
of homogeneous degree $\nu-1$.

\item
$\mathcal{H}^{\nu-1} = \Theta_{\Pi_j} \mathcal{L}^{N-1}$
on $\Pi_j$, where $\Theta_{\Pi_j} > 0$ is some constant,
$\mathcal{L}^{N-1}$ is a Lebesgue measure on $\Pi_j$.
\end{enumerate}
\end{lemma}

\begin{proof}
From the definition it is obvious that $L_j \cong \mathbb{R}$.
All three measures $\mathcal{H}^1$,
$\mathcal{H}^1_{\mathrm{riem}}$, $\mathcal{L}^1$ are
bi-invariant on $L_j$, hence, proportional.
The intersection of both Riemannian and sub-Riemannian balls
$B_r(0)$ with $L_j$ is an interval $(-r, r)$ and so
$\mathcal{H}^1|_{L_j} = \mathcal{H}^1_{\mathrm{riem}}|_{L_j}$.
Both measures coincide with $\mathcal{L}^1$ due to the
choice of a normalizing constant $\beta_1 = 2$.

Recall, that from~\eqref{EqGroupMultCoords}
$(x \cdot y)_j = x_j + y_j$ for $j = 1, \ldots, n_1$.
Hence, for $x, y \in \Pi_j$ (i.\,e.\ $x_j = y_j = 0$)
we have $xy, x^{-1} \in \Pi_j$ meaning $\Pi_j$ is a subgroup.
Next, for all $p \in \Pi_j$, $z \in \mathbb{G}$ we have
$(zpz^{-1})_j = 0$, i.\,e.
$\Pi_j$ is a normal subgroup.

It is easy to see that
$d(\delta_\lambda y) = \lambda^{\nu-1} dy$ on $\Pi_j$.
Since both measures $\mathcal{H}^{\nu-1}$
and $\mathcal{L}^{N-1}$ are left-invariant $\Pi_j$
and $\delta_r$-homogeneous of degree $\nu-1$,
similarly to Lemma~\ref{HausProps} we obtain
$\frac{\partial \mathcal{H}^{\nu-1}}{\partial \mathcal{L}^{N-1}} = \frac{\mathcal{H}^{\nu-1}(B_1 \cap \Pi_j)}{\mathcal{L}^{N-1}(B_1 \cap \Pi_j)}
= \Theta_{\Pi_j}$.
\end{proof}

Note, that $\Pi_j$ is always a graded nilpotent group but not
necessarily a Carnot group.

\begin{definition}
\label{DefProjection}
By the last Lemma for each
$j = 1, \ldots, n_1$ the group $\mathbb{G}$
is a semidirect product $\mathbb{G} = \Pi_j L_j$,
i.\,e.\ every $x \in \mathbb{G}$ is uniquely represented as
$x = p \exp(x_j X_j)$, $p \in \Pi_j$. Define the
\emph{projections}
$\mathrm{Pr}_j : \mathbb{G} \to \Pi_j$,
$\mathrm{pr}_j : \mathbb{G} \to \mathbb{R}$
as $\mathrm{Pr}_j(x) = p$,
$\mathrm{pr}_j(x) = x_j$.
In other words $\mathrm{Pr}_j$ is a projection on $\Pi_j$
along the integral lines of the vector field~$X_j$.
\end{definition}

The following properties are easy to check from the
definition.

\begin{lemma}[the properties of the projection]
\label{ProjProperties}
~

\begin{enumerate}

\item
$\mathrm{Pr}_j(\delta_r x) = \delta_r \mathrm{Pr}_j(x)$
for all $x \in \mathbb{G}$, $r > 0$.

\item
For all $x, y \in \mathbb{G}$
the identity holds
\[
  \mathrm{Pr}_j(x \cdot y) =
  \mathrm{Pr}_j(x) \cdot
  \exp(x_j X_j) \cdot
  \mathrm{Pr}_j(y) \cdot
  \exp(- x_j X_j).
\]

\end{enumerate}

\end{lemma}

In what follows we also use the following analytical properties
of projections.

\begin{lemma}
\label{LemmaProjJacobian}
$1)$
Let $s \in \mathbb{R}$. The mappings
$\Pi_{s,j} \to \Pi_j$ defined as
$x \mapsto \exp(-sX_j) x$,
$x \mapsto x \exp(-sX_j)$ and
$x \mapsto \mathrm{Pr}_j(x)$
are diffeomorphisms of $\Pi_{s,j}$ onto $\Pi_j$
and have a Jacobian of $1$.

$2)$ The Jacobi matrix of the mapping $\mathrm{Pr}_j$
in the basis $X_j, X_1, \ldots, X_{j-1}, X_{j+1}, \ldots, X_N$
(i.\,e.\ considering $x_j$ as the first coordinate)
has the form
\[
  D\mathrm{Pr}_j(x) =
  \begin{pmatrix}
  0 & T_j(x)
  \end{pmatrix},
\]
where $\det T_j(x) = 1$.
\end{lemma}

\begin{proof}
The restrictions of left and right shifts
$r_{s,j}(x) = x \cdot \exp(s X_j)$,
$l_{s,j}(x) = \exp(s X_j) \cdot x$ on the hyperplane
$\Pi_j$ are diffeomorphisms $\Pi_j \to \Pi_{s,j}$.
For the left shift we have
$(X_k l_{s,j})(x) = X_k (l_{s,j}(x))$,
$k = 1, \ldots, N$,
due to the left-invariance of
$X_1, \ldots, X_N$. Consequently, $Dl_{s,j} = \mathrm{Id}$.
For the right-shift on $\exp(s X_j)$ we have
\[
  (X_j r_{s,j})(x)
  = \frac{d}{dt} x \cdot \exp(s X_j) \cdot \exp(t X_j) \Big|_{t = 0}
  = \frac{d}{dt} x \cdot \exp((s + t)X_j) \Big|_{t = 0}
  = X_j (r_{s,j}(x)).
\]
Besides, $X_k r_{s,j}(x) \in T \Pi_{s,j}$
for $k \ne j$ and $x \in \Pi_{j}$.
Therefore, the Jacobi matrix of the mapping
$r_{s,j}$ in the basis
$X_j, X_1, \ldots, X_{j-1}, X_{j+1}, \ldots, X_N$
has the form
\[
  D r_{s,j}(x) =
  \begin{pmatrix}
  1 & 0 \\
  0 & DT_{s,j}(x)
  \end{pmatrix},
  \qquad
  x \in \Pi_{j},
\]
where $T_{s,j} : \Pi_j \to \Pi_{s,j}$ is the restriction of
$r_{s,j}$.
Since the Lebesgue measure $\mathcal{L}^N$ is right-invariant
on $\mathbb{G}$,
$\det DT_{s,j} = \det Dr_{s,j} = 1$.

For the mapping $\mathrm{Pr}_j|_{\Pi_{s,j}}$ it is enough
to notice that $\mathrm{Pr}_j(x) = x\exp(-sX_j)$
for $x \in \Pi_{s,j}$, i.\,e. the restriction of
$\mathrm{Pr}_j$ on $\Pi_{s,j}$ coincides with the restriction
of $r_{-s,j}$ on $\Pi_{s,j}$. P.\,1) of Lemma is proved.

Finally, for the mapping
$\mathrm{Pr}_j : \mathbb{G} \to \Pi_j$ we have
\[
  X_j \mathrm{Pr}_j(x)
  = \frac{d}{dt} \mathrm{Pr}_j(x \exp(t X_j)) \big|_{t = 0}
  = \frac{d}{dt} \mathrm{Pr}_j(x) \big|_{t = 0}
  = 0.
\]
Fix $s \in \mathbb{R}$.
For $y \in \Pi_{s,j}$ we have
$\mathrm{Pr}_j(y) = y \cdot \exp(-sX_j)$.
Thus, $\mathrm{Pr}_j|_{\Pi_{s,j}} = T^{-1}_{s,j}$,
where $T_{s,j} : \Pi_j \to \Pi_{s,j}$ is defined earlier.
It follows that the Jacobi matrix of $\mathrm{Pr}_j$
in the basis $X_j, X_1, \ldots, X_{j-1}, X_{j+1}, \ldots, X_N$
is
\[
  D\mathrm{Pr}_j(x) =
  \begin{pmatrix}
  0 & DT^{-1}_{s,j}(x)
  \end{pmatrix},
\]
and $\det DT^{-1}_{s,j}(x) = 1$.
P.\,2) of Lemma is proved.
\end{proof}

An analogue of Fubini's theorem may be immediately derived
from this lemma. A partial analogue of this theorem may also
be found in~\cite{Montefalcone2005}.

\begin{theorem}[Fubini type theorem]
\label{ThFubini}
If $f \in L_1(\mathbb{G})$, then

$1)$ for $\mathcal{H}^{\nu-1}$-a.\,e. $p \in \Pi_j$
the mapping $t \mapsto f(p \exp(t X_j))$ is integrable;

$2)$ for a.\,e. $t \in \mathbb{R}$
the mapping $p \mapsto f(p \exp(t X_j))$
$\mathcal{H}^{\nu-1}$-integrable;

$3)$ the following formula holds
\[
  \intop_{\Pi_j} d\mathcal{H}^{\nu-1}(p)
  \intop_{\mathbb{R}}
  f(p \exp(t X_j)) \, dt
  = \intop_{\mathbb{R}} dt \intop_{\Pi_j}
  f(p \exp(t X_j)) \, d\mathcal{H}^{\nu-1}(p)
  = \frac{\Theta_{\Pi_j}}{\Theta_\mathbb{G}}
    \intop_{\mathbb{G}} f(x) \, d\mathcal{H}^\nu(x).
\]
\end{theorem}

\begin{proof}
Applying the Riemannian coarea formula~\eqref{EqEuclideanCoarea}
to the smooth mapping $\mathrm{Pr}_j$ we obtain
\[
  \intop_{\Pi_j} dp
    \intop_{\mathrm{Pr}_j^{-1}(p)}
    f(y) \, d\mathcal{H}_{\mathrm{riem}}^1(y)
  =
  \intop_\mathbb{G} f(x) J(x, \mathrm{Pr}_j) \, dx.
\]
Wherein $J(x, \mathrm{Pr}_j)
= \sqrt{\det D\mathrm{Pr}_j(x) D\mathrm{Pr}_j^*(x)} = 1$
by p.\,2 of Lemma~\ref{LemmaProjJacobian}.
Since $\mathrm{Pr}_j^{-1}(p) = p L_j$,
by Lemma~\ref{LemmaSubgroups} and left-invariance of
$\mathcal{H}^1_{\mathrm{riem}}$ it follows
\[
  \intop_{\mathrm{Pr}_j^{-1}(p)}
  f(y) \, d\mathcal{H}_{\mathrm{riem}}^1(y)
  = \intop_{L_j}
  f(ph) \, d\mathcal{H}^1_{\mathrm{riem}}(h)
  = \intop_{\mathbb{R}}
  f(p\exp(tX_j)) \, dt.
\]
On the other hand, by Fubini's theorem
\[
  \intop_{\mathbb{G}} f(x) \, dx =
  \intop_{\mathbb{R}} dt
  \intop_{\Pi_{t,j}} f(z) \, dz.
\]
Since by p.\,1 of Lemma~\ref{LemmaProjJacobian}
the Jacobian of the change of variable $z = p \exp(t X_j)$
is equal to $1$, we have
\[
  \intop_{\Pi_{t,j}} f(z) \, dz
  = \intop_{\Pi_j} f(p\exp(t X_j)) \, dp.
\]
Thus, we obtain
\[
  \intop_{\Pi_j} dp \intop_{\mathbb{R}}
  f(p \exp(t X_j)) \, dt
  = \intop_{\mathbb{G}} f(x) \, dx
  = \intop_{\mathbb{R}} dt \intop_{\Pi_j}
  f(p \exp(t X_j)) \, dp.
\]
The statement of Theorem follows from the proportionality of
$dp$ and $d\mathcal{H}^{\nu-1}(p)$.
\end{proof}

\section{The coarea inequality for projections}

Let $\mathbb{G}, \widetilde{\mathbb{G}}$ be two Carnot groups.
The mapping $L : \mathbb{G} \to \widetilde{\mathbb{G}}$
is a \emph{contactomorphism} if it is a homomorphism
of Lie groups that is contact, i.\,e.\
$DL \langle H \mathbb{G} \rangle \subset H \widetilde{\mathbb{G}}$. The
corresponding mapping of Lie algebras, which for brevity
we denote as
\[
  \widehat{L} = \widetilde{\log} \circ L \circ \exp
\]
is a \emph{graded homomorphism of Lie algebras},
i.\,e. it is a linear mapping such that
$\widehat{L} \langle H \mathbb{G} \rangle \subset
H \widetilde{\mathbb{G}}$ and
$L \langle [X, Y] \rangle = [L \langle X \rangle, L \langle Y \rangle]$. In a graded basis such a mapping is represented
by a block-diagonal matrix.

The mapping $\varphi : E \to \widetilde{\mathbb{G}}$,
$E \subset \mathbb{G}$, is Pansu differentiable
($\mathcal{P}$-differentiable) at $x_0 \in E$ if there is
such a contactomorphism $L$ that
\[
  \widetilde{\rho}
  \bigl( \varphi(x_0)^{-1} \varphi(x), L(x_0^{-1} x) \bigr)
  = o(\rho(x_0, x))
  \quad \text{as } x \to x_0.
\]
We denote this homemorphism ($\mathcal{P}$-differential)
as $D_\mathcal{P} f(x_0)$ and the corresponding homemorphism
of Lie algebras as $\widehat{D} f(x_0)$.

Lipschitz continuous mappings of Carnot groups are
$\mathcal{P}$-differentiable a.\,e. In particular they are
contact. This is proved in~\cite{Pa1989} for mappings
defined on open sets and in~\cite{VodUkh1996} for mappings
defined on measurable sets.

For a linear mapping
$\widehat{L} : \mathfrak{g} \to \mathfrak{g}$
of Lie algrebra $\mathfrak{g}$ as
$\mathop{\mathrm{adj}} \widehat{L}$ we denote the adjoint
linear mapping. Its matrix is a transposed matrix of cofactors
and its main property is
\[
  \widehat{L} \circ \mathop{\mathrm{adj}} \widehat{L}
  = \mathop{\mathrm{adj}} \widehat{L} \circ \widehat{L}
  = \det \widehat{L} \cdot \mathrm{Id}.
\]
Note, that if $\widehat{L}$ is a graded homomorphism we also
have
$\mathop{\mathrm{adj}} \widehat{L} \langle H \mathbb{G}
\rangle \subset H \mathbb{G}$ due to block-diagonal structure
but $\mathop{\mathrm{adj}} \widehat{L}$ is generally not
a graded homomorphism.

The main result of this paper is the following

\begin{theorem}[coarea inequality]
\label{ThCoareaIneq}
Let $E \subseteq \mathbb{G}$,
$\varphi : E \to \mathbb{G}$ be Lipschitz continuous.
Then for every measurable $A \subseteq E$ the mapping
\[
  p \mapsto \mathcal{H}^1
  \big( (\mathrm{Pr}_j \circ \varphi)^{-1}(p) \cap A
  \big),
  \qquad
  p \in \Pi_j,
\]
is measurable and
\[
  \intop_{\Pi_j} \mathcal{H}^1
  \big( (\mathrm{Pr}_j \circ \varphi)^{-1}(p) \cap A
  \big) \, d\mathcal{H}^{\nu-1}(p)
  \le \frac{\Theta_{\Pi_j}}{\Theta_\mathbb{G}}
  \int\limits_A
  \bigl| \mathop{\mathrm{adj}} \widehat{D}\varphi(x) \langle X_j \rangle \bigr|
  \, d\mathcal{H}^\nu(x),
\]
where
$\Theta_{\mathbb{G}}$, $\Theta_{\Pi_j}$ are the constants of
proportionality of Hausdorff and Lebesgue measures, namely
$\mathcal{H}^\nu = \Theta_{\mathbb{G}} \mathcal{L}^N$
on $\mathbb{G}$,
$\mathcal{H}^{\nu-1} = \Theta_{\Pi_j} \mathcal{L}^{N-1}$
on $\Pi_j$.
\end{theorem}

\begin{definition}
Let $E \subseteq \mathbb{G}$, $x \in E$ be a point of
$\mathcal{P}$-differentiability of the mapping
$\varphi : E \to \mathbb{G}$.
Call the \emph{sub-Riemannian Jacobian} of the mapping
$\mathrm{Pr}_j \circ \varphi$ at $x$ the value
\[
  \mathcal{J}(x, \mathrm{Pr}_j \circ \varphi)
  =
  \frac{\Theta_{\Pi_j}}{\Theta_\mathbb{G}}
  \bigl| \mathop{\mathrm{adj}}
  \widehat{D}\varphi(x) \langle X_j \rangle \bigr|.
\]
\end{definition}

The reasonableness of such a definition follows from the
following lemma.

\begin{lemma}
Let $L$ be a contactomorphism of the group $\mathbb{G}$, then
\[
\frac{1}{\mathcal{H}^\nu(\overline{B}_1)}
  \int\limits_{\Pi_j}
  \mathcal{H}^1 \bigl(
  (\mathrm{Pr}_j \circ L)^{-1}(y) \cap \overline{B}_1 \bigr)
  \, d\mathcal{H}^{\nu-1}(y)
  =
  \frac{\Theta_{\Pi_j}}{\Theta_\mathbb{G}}
  \bigl| \mathop{\mathrm{adj}} \widehat{L} \langle X_j \rangle \bigr|.
\]
\end{lemma}

\begin{proof}
Applying the Riemannian coarea formula~\eqref{EqEuclideanCoarea}
to the smooth mapping
$w \mapsto \mathrm{Pr}_j(L(w))$ on a closed ball
$\overline{B}_1$ we obtain
\begin{equation}
\label{EqRiemCoarea}
  \intop_{\Pi_j}
  \mathcal{H}^1_{\mathrm{riem}}
  \big( (\mathrm{Pr} \circ L)^{-1}(y)
  \cap \overline{B}_1  \big) \, dy
  = \intop_{\overline{B}_1}
  J(x, \mathrm{Pr}_j \circ L) \, dx,
\end{equation}
where $\mathcal{H}^1_{\mathrm{riem}}$ is a Hausdorff measure
with respect to the Riemannian metric and
\[
  J(x, \mathrm{Pr}_j \circ L)
  = \sqrt{\det D(\mathrm{Pr} \circ L) D(\mathrm{Pr} \circ L)^*}.
\]

Show that $J(x, \mathrm{Pr}_j \circ L) = \bigl|
\mathop{\mathrm{adj}} \widehat{L} \langle X_j \rangle \bigr|$.
By p.\,2 of Lemma~\ref{LemmaProjJacobian} in the basis
of vector fields
$X_j, X_1, \ldots, X_{j-1}, X_{j+1}, \ldots, X_N$ we have
\[
  D(\mathrm{Pr}_j \circ L)(x)
  =
  \begin{pmatrix}
  0 & T_j(L(x))
  \end{pmatrix} \widehat{L}
  = T_j(L(x)) \widetilde{L}_j,
\]
where $\widetilde{L}_j$ is an $(N-1) \times N$ matrix obtained
by removal of the first row (corresponding to $X_j$)
from $\widehat{L}$. From this
\[
  J(\cdot, \mathrm{Pr}_j \circ L)
  = \sqrt{\det T_j \widetilde{L}_j \widetilde{L}_j^* T_j^*}
  = \left| \det T_j \right|
    \sqrt{\det \widetilde{L}_j \widetilde{L}_j^*}.
\]
By Lemma~\ref{LemmaProjJacobian} $|\det T_j| = 1$.
By Binet--Cauchy formula
\[
  \det \widetilde{L}_j \widetilde{L}_j^*
  = \sum_{i=1}^N (\det \widetilde{L}_{ji})^2
  = \bigl|
      \mathop{\mathrm{adj}} \widehat{L} \langle X_j \rangle
    \bigr|^2,
\]
where $\widetilde{L}_{ji}$ are $(N-1) \times (N-1)$-minors
of the matrix $\widetilde{L}_j$ with $i$-th column excluded.
Note, that $\mathop{\mathrm{adj}} \widehat{L}$
is block-diagonal in the basis of $X_1, \ldots, X_N$
since $\widehat{L}$ is such. Hence,
$\mathop{\mathrm{adj}} \widehat{L} \langle X_j \rangle \in H\mathbb{G}$.

Now prove that
$\mathcal{H}^1_{\mathrm{riem}}
  \big( (\mathrm{Pr} \circ L)^{-1}(y)
  \cap \overline{B}_1 \big)
  =
  \mathcal{H}^1
  \big( (\mathrm{Pr} \circ L)^{-1}(y)
  \cap \overline{B}_1
  \big)$
for a.\,e. $y \in \Pi_j$.

\textbf{Case 1.} Let $\det \widehat{L} \ne 0$.
Then
$\widehat{L}^{-1} \langle X \rangle =
(\det \widehat{L})^{-1} \mathop{\mathrm{adj}} \widehat{L} \langle X \rangle$.
Since $\mathrm{Pr}_j^{-1}(y)$ is an integral curve of the
vector field $X_j$ passing through $(0,y)$,
its preimage $L^{-1}(\mathrm{Pr}_j^{-1}(y))$ is an integral curve
of the horizontal vector field
$\mathop{\mathrm{adj}} \widehat{L} \langle X_j \rangle$.
Therefore,
\[
  \mathcal{H}^1_{\mathrm{riem}}
  \big( (\mathrm{Pr} \circ L)^{-1}(y)
  \cap \overline{B}_1 \big)
  =
  \mathcal{H}^1
  \big( (\mathrm{Pr} \circ L)^{-1}(y)
  \cap \overline{B}_1
  \big).
\]

\textbf{Case 2.}
Let $\det \widehat L = 0$ but
$\mathop{\mathrm{adj}} \widehat{L} \langle X_j \rangle \ne 0$.
In this case $\dim \widehat{L} \langle V_1 \rangle = n_1 - 1$,
$j$-th row of the matrix $\widehat{L}$ is linearly dependent
with the rows $1, \ldots, j-1, j+1, \ldots, n_1$, and
$X_j \not\in \widehat{L} \langle V_1 \rangle$.

Note, that for every $y \in \Pi_j$ the intersection
$\mathrm{Pr}_j^{-1}(y) \cap L(\mathbb{G})$ consists of no more
than a single point. Indeed,
$L(\mathbb{G})$ is a subgroup of $\mathbb{G}$.
If $y \exp(aX_j), y \exp(b X_j) \in L(\mathbb{G})$ then
\[
  \big( y \exp(aX_j) \big)^{-1} y \exp(b X_j)
  = \exp ((b-a)X_j) \in L(\mathbb{G}).
\]
But from $\exp(X_j) \not\in L(\mathbb{G})$ it follows $a = b$.
Since
$\mathop{\mathrm{adj}} \widehat{L} \langle X_j \rangle
\in \ker \widehat{L}$ the preimage of the point
$\mathrm{Pr}_j^{-1}(y) \cap L(\mathbb{G})$
under $L$ is an integral curve of a horizontal vector field
$\mathop{\mathrm{adj}} \widehat{L} \langle X_j \rangle$.
Hence,
\[
  \mathcal{H}^1_{\mathrm{riem}}
  \big( (\mathrm{Pr} \circ L)^{-1}(y)
  \cap \overline{B}_1 \big)
  =
  \mathcal{H}^1
  \big( (\mathrm{Pr} \circ L)^{-1}(y)
  \cap \overline{B}_1
  \big).
\]

\textbf{Case 3.}
Let
$\mathop{\mathrm{adj}} \widehat{L} \langle X_j \rangle = 0$.
Then the rows $1, \ldots, j-1, j+1, \ldots, n_1$
of the matrix~$\widehat{L}$ are linearly dependent and there is
$Y \in \mathrm{span} \{ X_1, \ldots, X_{j-1}, X_{j+1}, \ldots, X_{n_1} \}$ such that $L \langle V_1 \rangle \perp Y$.
Therefore, $\mathrm{Pr}_j(L(\mathbb{G}))$ is in a subspace
of $\Pi_j$ of dimension less than $N-1$ and
$(\mathrm{Pr} \circ L)^{-1}(y)\cap \overline{B}_1 = \varnothing$
for $\mathcal{L}^{N-1}$-a.\,e.\ $y \in \Pi_j$. Therefore,
\[
  \mathcal{H}^1_{\mathrm{riem}}
  \big( (\mathrm{Pr} \circ L)^{-1}(y)
  \cap \overline{B}_1 \big)
  =
  \mathcal{H}^1 \big( (\mathrm{Pr} \circ L)^{-1}(y)\cap \overline{B}_1 \big) = 0
\]
for a.\,e.\ $y \in \Pi_j$.

Hence, the equality~\eqref{EqRiemCoarea} becomes
\[
  \intop_{\Pi_j}
  \mathcal{H}^1
  \big( (\mathrm{Pr} \circ L)^{-1}(y)
  \cap \overline{B}_1  \big) \, dy
  = \intop_{\overline{B}_1}
  \bigl|
      \mathop{\mathrm{adj}} \widehat{L} \langle X_j \rangle
    \bigr| \, dx
  = \mathcal{L}^N(\overline{B}_1) \,
  \bigl|
      \mathop{\mathrm{adj}} \widehat{L} \langle X_j \rangle
  \bigr|.
\]
The statement of Lemma now follows from the proportionality
of Lebesgue and Hausdorff measures.
\end{proof}

\begin{lemma}
\label{LemmaCoareaChange}
Let $E \subseteq \mathbb{G}$,
$\varphi : E \to \mathbb{G}$ be Lipschitz continuous
and $x_0 \in E$.
Then for every $\varepsilon > 0$
\[
  \mathop{\int^*}_{\Pi_j} \mathcal{H}^1_\varepsilon
  \big[
  (\mathrm{Pr}_j \circ \varphi)^{-1}(y) \big]
  \, d\mathcal{H}^{\nu-1}(y)
  =
  \mathop{\int^*}_{\Pi_j} \mathcal{H}^1_\varepsilon
  \big\{ x \in E :
  \mathrm{Pr}_j (\varphi(x_0)^{-1} \varphi(x)) = z
  \big\}
  \, d\mathcal{H}^{\nu-1}(z).
\]
\end{lemma}

\begin{proof}
Let $s_0 = \mathrm{pr}_j(x_0)$.
By the properties of $\mathrm{Pr}_j$
(see Lemma~\ref{ProjProperties}) we have
\[
  \mathrm{Pr}_j(\varphi(x))
  = \mathrm{Pr}_j
  (\varphi(x_0) \varphi(x_0)^{-1} \varphi(x))
  = \mathrm{Pr}_j(\varphi(x_0))
  \cdot \gamma_j(s_0)
  \cdot \mathrm{Pr}_j(\varphi(x_0)^{-1}\varphi(x))
  \cdot \gamma_j(-s_0).
\]
Therefore,
\[
  \big\{ x \in E :
  \mathrm{Pr}_j(\varphi(x)) = y \big\}
  = \big\{ x \in E :
  \mathrm{Pr}_j(\varphi(x_0)^{-1}\varphi(x)) = z \big\},
\]
where
$z = \mathrm{Pr}_j(\varphi(x_0)) \, \gamma_j(s_0) \, y \, \gamma_j(-s_0)$. By Lemma~\ref{LemmaProjJacobian} the mapping
$y \mapsto z$ is a diffeomorphism of $\Pi_j$ on $\Pi_j$
with a Jacobian of $1$.
The statement of Lemma follows.
\end{proof}

The following lemma is an analogue of the well-known
Eilenberg's inequality~\cite{Eil, FedEil}
which is a key ingredient of most proofs of the classical
coarea formula. An analogue of this inequality is even
known for Lipschitz mappings of metric spaces
(see review and self-contained proof in~\cite{EsmHaj}).
These results, however, are not applicable to our case
since the composition $\mathrm{Pr}_j \circ \varphi$
is not Lipschitz continuous.

\begin{lemma}[Eilenberg type inequality]
\label{LemmaCoareaAC}
Let $E \subseteq \mathbb{G}$,
$\varphi : E \to \mathbb{G}$ be Lipschitz continuous. Then
\[
  \mathop{\int^*}_{\Pi_j} \mathcal{H}^1 \big(
  {(\mathrm{Pr}_j \circ \varphi)^{-1}(y)} \big)
  d\mathcal{H}^{\nu-1}(y)
  \leq C \, \mathrm{Lip}(\varphi)^{\nu-1}
  \,\mathcal{H}^{\nu}(E),
\]
where the constant $C$ is independent of the choice
of $E$ and $\varphi$.
\end{lemma}

\begin{proof}
The formula is trivial for the set of infinite measure
so we assume $\mathcal{H}^\nu(E) < +\infty$.

Fix $\delta, \varepsilon > 0$.
There is a countable family of balls
$B_k = B_{r_k}(x_k)$, $k \in \mathbb{N}$,
such that
\[
  E \subseteq \bigcup_{k \in \mathbb{N}} B_k, \quad
  \sup_{k \in \mathbb{N}} r_k \le \delta
  \quad \text{and} \quad
  \beta_\nu \sum_{k \in \mathbb{N}} r_k^\nu
  < \mathcal{H}^\nu(E) + \varepsilon.
\]
We may assume without loss of generality that every ball
intersects with $E$.
Let $\widetilde{x}_k \in B_k \cap E$ and
$\widetilde{B}_k = B_{2 r_k}(\widetilde{x}_k)$.
Then
$B_k \subset \widetilde{B}_k$ for all $k \in \mathbb{N}$.
By Lemma~\ref{LemmaCoareaChange}
\[
  \mathop{\int^*}_{\Pi_j} \mathcal{H}^1_{2\delta}
  \big\{ x \in B_k \cap E :
  \mathrm{Pr}_j (\varphi(x)) = y
  \big\}
  \, dy
  \le
  \mathop{\int^*}_{\Pi_j} \mathcal{H}^1_{2\delta}
  \big\{ x \in \widetilde{B}_k \cap E :
  \mathrm{Pr}_j (\varphi(\widetilde{x}_k)^{-1} \varphi(x)) = z
  \big\}
  \, dz.
\]
Since
$\rho(\varphi(\widetilde{x}_k), \varphi(x))
\le \mathrm{Lip}(\varphi) \rho(\widetilde{x}_k, x)
\le 2 \, \mathrm{Lip}(\varphi) r_k$, the function under the last
integral vanishes as soon as
$z \not\in \mathrm{Pr}_j \big( B_{2 \, \mathrm{Lip}(\varphi) r_k}(0) \big)$. Therefore,
\[
  \mathop{\int^*}_{\Pi_j} \mathcal{H}^1_{2\delta}
  \big\{ x \in \widetilde{B}_k \cap E :
  \mathrm{Pr}_j (\varphi(\widetilde{x}_k)^{-1} \varphi(x)) = z
  \big\}
  \, d\mathcal{H}^{\nu-1}(z)
  \le
  \intop_{\mathrm{Pr}_j ( B_{2 \, \mathrm{Lip}(\varphi) r_k})}
  \mathcal{H}^1_{2\delta} ( \widetilde{B}_k )
  \, d\mathcal{H}^{\nu-1}(z).
\]
Since $\mathrm{Pr}_j$ and $\mathcal{H}^{\nu-1}$
are $\delta_r$-homogeneous, it follows
\[
  \mathcal{H}^{\nu-1}\big(
  \mathrm{Pr}_j (B_{2 \, \mathrm{Lip}(\varphi) r_k})
  \big)
  = \mathcal{H}^{\nu-1}\big(
  \delta_{2 \, \mathrm{Lip}(\varphi) r_k}
  \mathrm{Pr}_j (B_1)
  \big)
  = (2 \, \mathrm{Lip}(\varphi) r_k)^{\nu-1}
  \mathcal{H}^{\nu-1}\big( \mathrm{Pr}_j (B_1) \big).
\]
From here, taking into account that
$\mathcal{H}^1_{2\delta} ( \widetilde{B}_k ) \le
2 \beta_1 r_k$,
we have
\[
  \mathop{\int^*}_{\mathrm{Pr}_j ( B_{2 \, \mathrm{Lip}(\varphi) r_k})}
  \mathcal{H}^1_{2\delta} ( \widetilde{B}_k )
  \, d\mathcal{H}^{\nu-1}(z)
  \le
  2^\nu \beta_1
  \mathcal{H}^{\nu-1}\big( \mathrm{Pr}_j (B_1) \big)
  \mathrm{Lip}(\varphi)^{\nu-1}
  r_k^\nu
  =
  C \,
  \mathrm{Lip}(\varphi)^{\nu-1}
  r_k^\nu.
\]
Since for every $y \in \Pi_j$ it holds
\[
  \mathcal{H}^1_{2\delta}
  \big\{ x \in E :
  \mathrm{Pr}_j (\varphi(x)) = y
  \big\}
   \le \sum_{k=1}^\infty
  \mathcal{H}^1_{2\delta}
  \big\{ x \in B_k \cap E :
  \mathrm{Pr}_j (\varphi(x)) = y
  \big\},
\]
we derive
\[
  \mathop{\int^*}_{\Pi_j} \mathcal{H}^1_{2\delta}
  \big(
  (\mathrm{Pr}_j \circ \varphi)^{-1}(y)
  \big)
  \, d\mathcal{H}^{\nu-1}(y)
  \le \sum_{k=1}^\infty
  C \,
  \mathrm{Lip}(\varphi)^{\nu-1}
  r_k^\nu
  \le
  \frac{C}{\beta_\nu}
  \mathrm{Lip}(\varphi)^{\nu-1}
  [\mathcal{H}^\nu(E) + \varepsilon].
\]
Finally, since $\mathcal{H}^1_{2\delta}$ are monotone in
$\delta$, from Fatou's Lemma we obtain
\[
  \mathop{\int^*}_{\Pi_j} \mathcal{H}^1
  \big(
  (\mathrm{Pr}_j \circ \varphi)^{-1}(y)
  \big)
  \, d\mathcal{H}^{\nu-1}(y)
  \le
  \frac{C}{\beta_\nu}
  \mathrm{Lip}(\varphi)^{\nu-1}
  \big[ \mathcal{H}^\nu(E) + \varepsilon].
\]
The statement of Lemma follows from the arbitrary choice of
$\varepsilon > 0$.
\end{proof}

\begin{corollary}
\label{CorrMeasurable}
Let $E \subseteq \mathbb{G}$ be closed and
$\varphi : E \to \mathbb{G}$ be Lipschitz continuous.

\begin{enumerate}

\item
For each compact $K \subset \mathbb{G}$ and $\varepsilon > 0$
the function $g_{j,\varepsilon} : \Pi_j \to [0, +\infty)$
defined as
\[
  g_{j,\varepsilon}(y) = \mathcal{H}^1_\varepsilon
  \big( (\mathrm{Pr}_j \circ \varphi)^{-1}(y) \cap K \big),
\]
is measurable.

\item
For each measurable $A \subset \mathbb{G}$
the function $g_j : \Pi_j \to [0, +\infty]$
defined as
\[
  g_j(y) = \mathcal{H}^1
  \big( (\mathrm{Pr}_j \circ \varphi)^{-1}(y) \cap A \big)
\]
is measurable.

\end{enumerate}
\end{corollary}

\begin{proof}
\textbf{Step 1.}
Let $K$ be a compact set.
Fix $\varepsilon > 0$. Consider a set
\[
  U_\varepsilon(t) = \{ y \in \Pi_j : g_{j,\varepsilon}(y) < t \},
  \quad t > 0.
\]
Let $y \in U_\varepsilon(t)$. Since
$(\mathrm{Pr}_j \circ \varphi)^{-1}(y) \cap K$ is a compact,
there is a finite family of balls $B_k = B_{r_k}(x_k)$, $k = 1, \ldots, m$ such that
\[
  (\mathrm{Pr}_j \circ \varphi)^{-1}(y)
  \subset \bigcup_{k=1}^m B_k,
  \qquad
  r_k \le \varepsilon,
  \qquad
  \beta_1 \sum_{k=1}^m r_k < t.
\]
Since $\mathrm{Pr}_j \circ \varphi$ is continuous and $K$
is compact, $(\mathrm{Pr}_j \circ \varphi)^{-1}(z) \cap A
\subset \bigcup_{k=1}^m B_k$ for $z$ belonging to some
neighborhood of $y$. Thus, $U_\varepsilon(t)$ is open
and $g_{j,\varepsilon}$ is measurable. Hence,
$g_j = \lim\limits_{\varepsilon \to 0} g_{j,\varepsilon}$
is also measurable.

\textbf{Step 2.} Let $A$ be open. Then there exists a sequence
of compacts $K_n$, $n \in \mathbb{N}$,
such that $K_1 \subseteq K_2 \subseteq \ldots \subset A$ and
$\bigcup_{n=1}^\infty K_n = A$. Therefore, the function
\[
  g_j(y) = \mathcal{H}^1
  \big( (\mathrm{Pr}_j \circ \varphi)^{-1}(y) \cap A \big)
  = \lim_{n \to \infty} \mathcal{H}^1
  \big( (\mathrm{Pr}_j \circ \varphi)^{-1}(y) \cap K_n \big),
\]
is measurable.

\textbf{Step 3.}
Let $\mathcal{H}^\nu(A) < \infty$.
There is a sequence of open sets $U_n$, $n \in \mathbb{N}$,
such that $U_1 \supseteq U_2 \supseteq \ldots \supset A$,
$\mathcal{H}^\nu(U_1) < +\infty$, and
\[
  \lim\limits_{n \to \infty} \mathcal{H}^\nu(U_n \setminus A) = 0.
\]
By Lemma~\ref{LemmaCoareaAC}
\[
  \intop_{\Pi_j} \mathcal{H}^1
  \big( (\mathrm{Pr}_j \circ \varphi)^{-1}(y)
  \cap (U_n) \big) d\mathcal{H}^{\nu-1}(y)
  \le C \, \mathrm{Lip}(\varphi)^{\nu-1}
  \mathcal{H}^\nu(U_n) < +\infty.
\]
Therefore, the function
$y \mapsto \mathcal{H}^1
\big( (\mathrm{Pr}_j \circ \varphi)^{-1}(y)
\cap (U_n) \big)$ is finite a.\,e.
Then the sequence of functions
\[
  \mathcal{H}^1
  \big( (\mathrm{Pr}_j \circ \varphi)^{-1}(y)
  \cap (U_n \setminus A) \big)
  =
  \mathcal{H}^1
  \big( (\mathrm{Pr}_j \circ \varphi)^{-1}(y)
  \cap U_n \big)
  -
  \mathcal{H}^1
  \big( (\mathrm{Pr}_j \circ \varphi)^{-1}(y)
  \cap A \big)
\]
is finite a.\,e., decreases monotonically and by
Lemma~\ref{LemmaCoareaAC}
\[
  \mathop{\int^*}_{\Pi_j}
    \mathcal{H}^1
  \big( (\mathrm{Pr}_j \circ \varphi)^{-1}(y)
  \cap (U_n \setminus A) \big) \,
  d\mathcal{H}^{\nu-1}(y)
  \le C \, \mathrm{Lip}(\varphi)^{\nu-1}
  \mathcal{H}^\nu(U_n \setminus A) \to 0
\]
as $n \to \infty$. It follows
\[
  \mathcal{H}^1
  \big( (\mathrm{Pr}_j \circ \varphi)^{-1}(y)
  \cap U_n \big)
  \to
  \mathcal{H}^1
  \big( (\mathrm{Pr}_j \circ \varphi)^{-1}(y)
  \cap A \big)
\]
as $n \to \infty$ for a.\,e. $y \in \Pi_j$.

\textbf{Step 4.}
Finally, if $\mathcal{H}^\nu(A) = +\infty$
it is enough to split $A$ into a countable disjoint
family of measurable sets of finite measure.
The lemma is proved.
\end{proof}

The proof of the following Lemma is inpired in part
by the reasoning in~\cite{Magn2002}.

\begin{lemma}[measure density estimate]
\label{LemmaEpsMeasureDensity}
Let $E \subseteq \mathbb{G}$ be closed and
$\varphi : E \to \mathbb{G}$ be Lipschitz continuous.
Define an outer measure $\mu_{\varepsilon,j}$,
$\varepsilon > 0$,
on subsets
$F \subseteq \mathbb{G}$ as
\[
  \mu_{\varepsilon,j}(F)
  = \mathop{\int^*}_{\mathrm{Pr}_j(\mathbb{G})}
  \mathcal{H}^1_\varepsilon
  \big(
  (\mathrm{Pr}_j \circ \varphi)^{-1}(y) \cap E \cap F
  \big) \, d\mathcal{H}^{\nu-1}(y).
\]
Then for $\mathcal{H}^\nu$-a.\,e.\ $x \in E$ we have
\[
  \limsup_{r \to 0}
  \frac{\mu_{\varepsilon,j}
  \big( \overline{B_r}(x) \big)}
  {\mathcal{H}^\nu
  \big( \overline{B_r}(x) \big)}
  \le
  \mathcal{J}(x, \mathrm{Pr} \circ \varphi).
\]
\end{lemma}

\begin{proof}
Let $x_0 \in E$ be a density point of~$E$ and
$\mathcal{P}$-\hspace{0pt}differentiability point of~$\varphi$
(almost all points of $E$ are such).
By Lemma~\ref{LemmaCoareaChange}
\begin{multline*}
  \mu_{\varepsilon,j} \big( \overline{B_r}(x_0) \big)
  =
  \int\limits_{\Pi_j} \mathcal{H}^1_\varepsilon
  \big\{ x \in \overline{B_r}(x_0) \cap E :
  \mathrm{Pr}_j (\varphi(x)) = y
  \big\} \, d\mathcal{H}^{\nu-1}(y)
  \\ =
  \int\limits_{\Pi_j} \mathcal{H}^1_\varepsilon
  \big\{ x \in \overline{B_r}(x_0) \cap E :
  \mathrm{Pr}_j (\varphi(x_0)^{-1}\varphi(x)) = z
  \big\} \, d\mathcal{H}^{\nu-1}(z).
\end{multline*}
Since $\mathrm{Pr}_j$ is $\delta_r$-homogeneous we have
\[
  \mathrm{Pr}_j (\varphi(x_0)^{-1} \varphi(x)) = z
  \quad \Leftrightarrow \quad
  \mathrm{Pr}_j \big[ \delta_{\frac 1r} \big(
  \varphi(x_0)^{-1} \varphi(x)
  \big) \big] = \delta_{\frac 1r} z.
\]
Applying a change of variable
$z = \delta_r w$ (in this case $dz = r^{\nu-1} dw$)
to the last integral we obtain
\begin{align}
  &
  \int\limits_{\Pi_j} \mathcal{H}^1_\varepsilon
  \big\{ x \in \overline{B_r}(x_0) \cap E :
  \mathrm{Pr}_j (\varphi(x_0)^{-1} \varphi(x)) = z
  \big\} \, d\mathcal{H}^{\nu-1}(z)
  \notag
  \\ & \qquad =
  \int\limits_{\Pi_j} \mathcal{H}^1_\varepsilon
  \big\{ x \in \overline{B_r}(x_0) \cap E :
  \mathrm{Pr}_j \big[ \delta_{\frac 1r} \big(
  \varphi(x_0)^{-1} \varphi(x)
  \big) \big] = w
  \big\} r^{\nu-1}\, d\mathcal{H}^{\nu-1}(w)
  \notag
  \\ & \qquad = r^{\nu-1}
  \int\limits_{\Pi_j} \mathcal{H}^1_\varepsilon
  \big\{ x \in \overline{B_r}(x_0) \cap E :
  \delta_{\frac 1r} \big(
  \varphi(x_0)^{-1} \varphi(x)
  \big) \in \mathrm{Pr}_j^{-1}(w)
  \big\} \, d\mathcal{H}^{\nu-1}(w).
\label{EqEst1}
\end{align}
Next, consider a change of variable $x = x_0 \delta_r \xi$,
$\xi \in \overline{B_1}$.
From the properties of $\mathcal{H}^1_\varepsilon$
(see Lemma~\ref{EpsHausProps})
for $r \in (0, 1]$ we have
\begin{multline}
\label{EqEst2}
  \mathcal{H}^1_\varepsilon
  \big\{ x \in \overline{B_r}(x_0) \cap E :
  \delta_{\frac 1r} \big(
  \varphi(x_0)^{-1} \varphi(x)
  \big) \in \mathrm{Pr}_j^{-1}(w)
  \big\}
  \\ \le r
  \mathcal{H}^1_\varepsilon
  \big\{ \xi \in \overline{B_1} \cap \delta_{\frac 1r} (x_0^{-1} E) :
  \delta_{\frac 1r} \big(
  \varphi(x_0)^{-1} \varphi(x_0 \delta_r \xi)
  \big) \in \mathrm{Pr}_j^{-1}(w)
  \big\}
  .
\end{multline}
Since $\varphi$ is $\mathcal{P}$-differentiable at $x_0$
it follows
\begin{equation}
\label{EqDiffInCoareaProof}
  \rho \bigl( \delta_{\frac 1r} \bigl(
  \varphi(x_0)^{-1} \varphi(x_0 \delta_r \xi)
  \bigr), L(\xi) \bigr) \le \beta(r),
\end{equation}
where $L = \widehat{D}\varphi(x_0)$ and $\beta(r) \to 0$
monotonically as $r \to 0$. Introduce for each
$w \in \Pi_j$ a family of compact sets
\[
  K_r(w) = \bigl\{ \xi \in \overline{B_1} :
  \mathrm{dist} \bigl(
  L(\xi), \mathrm{Pr}_j^{-1}(w) \bigr) \le \beta(r)
  \bigr\},
  \qquad 0 < r \le 1.
\]
Then $K_{r_1}(w) \subseteq K_{r_2}(w)$ as long as
$0 < r_1 \le r_2 \le 1$ and
\[
  \bigcap_{r > 0} K_r(w) =
  \overline{B_1} \cap (\mathrm{Pr}_j \circ L)^{-1}(w).
\]
From Lemma~\ref{EpsHausProps} we derive
$\mathcal{H}^1_{\varepsilon}(K_r(w)) \to
\mathcal{H}^1_{\varepsilon} \big( \overline{B_1} \cap (\mathrm{Pr}_j \circ L)^{-1}(w) \big)$ as $r \to 0$.
Besides, from~\eqref{EqDiffInCoareaProof}
\[
  \big\{ \xi \in \overline{B_1} \cap \delta_{\frac 1r} (x_0^{-1} E) :
  \delta_{\frac 1r} \big(
  \varphi(x_0)^{-1} \varphi(x_0 \delta_r \xi)
  \big) \in \mathrm{Pr}_j^{-1}(w)
  \big\} \subset K_r(w).
\]
Hance, taking into account~\eqref{EqEst1} and \eqref{EqEst2}
\begin{equation}
\label{EqEst3}
  \frac{1}{r^\nu}
  \int\limits_{\Pi_j} \mathcal{H}^1_\varepsilon
  \big\{ x \in \overline{B_r}(x_0) \cap E :
  \mathrm{Pr}_j \circ \varphi(x) = y
  \big\} \, d\mathcal{H}^{\nu-1}(y)
  \le
  \int\limits_{\Pi_j}
  \mathcal{H}^1_\varepsilon(K_r(w))
  \, d\mathcal{H}^{\nu-1}(w).
\end{equation}
Since $K_1(w) \ne \varnothing$ only for $w$ belonging to
some neighborhood of the origin $\Pi_j \cap B_a$,
the functions under the last integral have a common integrable
dominant
\[
  \mathcal{H}^1_\varepsilon(K_r(w))
  \le
  \mathcal{H}^1_\varepsilon(\overline{B_1})
  \chi_{\Pi_j \cap B_a}(w),
  \qquad r \in (0, 1].
\]
Therefore, going to the limit in~\eqref{EqEst3}
as $r \to 0$ we obtain
\[
  \limsup_{r \to 0}
  \frac{\mu_{\varepsilon,j}
  \big( \overline{B_r}(x_0) \big)}{r^\nu}
  \le
  \int\limits_{\Pi_j}
  \mathcal{H}^1_{\varepsilon} \big( \overline{B_1} \cap (\mathrm{Pr}_j \circ L)^{-1}(w) \big)
  \, d\mathcal{H}^{\nu-1}(w).
\]
Finally, since
$\mathcal{H}^\nu(\overline{B_r}(x_0)) =
\mathcal{H}^\nu(\overline{B_1}) r^\nu$
and $\mathcal{H}^1_\varepsilon \le \mathcal{H}^1$
we derive
\[
  \limsup_{r \to 0}
  \frac{\mu_{\varepsilon,j}
  \big( \overline{B_r}(x_0) \big)}
  {\mathcal{H}^\nu
  \big( \overline{B_r}(x_0) \big)}
  \le \frac{1}{\mathcal{H}^\nu(\overline{B_1})}
  \int\limits_{\Pi_j}
  \mathcal{H}^1 \big( \overline{B_1} \cap (\mathrm{Pr}_j \circ L)^{-1}(w) \big)
  \, d\mathcal{H}^{\nu-1}(w)
  =
  \mathcal{J}(x_0, \mathrm{Pr} \circ \varphi),
\]
which is precisely the statement of Lemma.
\end{proof}

\begin{proof}[Proof of Theorem~\ref{ThCoareaIneq}]
Note that the Lipschitz continuous
mapping $\varphi$ may be extended to the closure
$\overline{E}$ while preserving the Lipschitz constant.
Also, if $\mathcal{H}^\nu(E) = +\infty$ it is enough to split $E$
into a countable family of disjoint subsets and prove the
inequality on each such subset.
So in the following we assume that $E$ is closed and
$\mathcal{H}^\nu(E) < +\infty$.

The mapping
$x \mapsto \mathcal{J}(x, \mathrm{Pr} \circ \varphi)$
is measurable since the partial derivatives
$X_1 \varphi, \ldots, X_{n_1} \varphi$ are measurable.
By Corollary~\ref{CorrMeasurable} the measure
\[
  \mu_j(A) =
  \intop_{\mathrm{Pr}_j(\mathbb{G})} \mathcal{H}^1
  \big( (\mathrm{Pr}_j \circ \varphi)^{-1}(y) \cap A
  \big) \, d\mathcal{H}^{\nu-1}(y)
\]
is defined on all measurable
$A \subseteq \mathbb{G}$.

\textbf{Step 1.}
Let $V \subseteq A$ be measurable and such that
$\mathcal{J}(x, \mathrm{Pr} \circ \varphi) < c$
for a.\,e.\ $x \in V$.
Prove that $\mu_j(V) \le c \mathcal{H}^\nu(A)$.

Fix $\delta > 0$. There is open
$U \supseteq V$ such that
$\mathcal{H}^\nu(U \setminus V) < \delta$.
Fix $\varepsilon > 0$.
By Lemma~\ref{LemmaEpsMeasureDensity}
there exists a null set $\Sigma_1 \subset V$
such that for each $x \in V \setminus \Sigma_1$ there is
$r_\varepsilon(x) > 0$ where
\[
  \frac{\mathcal{\mu}_{\varepsilon, j}
  (\overline{B_r}(x))}
  {\mathcal{H}^\nu(\overline{B_r}(x))} < c,
  \quad 0 < r \le r_\varepsilon(x).
\]
Consider a family of balls $\overline{B_r}(x)$
such that
$x \in V \setminus \Sigma_1$,
$\overline{B_r}(x) \subseteq U$,
$0 < r \le r_\varepsilon$. By Vitali's Theorem there is
a countable disjoint subset
$\{ \overline{B_{r_k}}(x_k) \}$ that covers
$V \setminus \Sigma_1$ up to a null set
$\Sigma_2$.
By Lemma~\ref{LemmaCoareaAC} we have
$\mu_{\varepsilon,j}(\Sigma_i) \le \mu_{j}(\Sigma_i) = 0$, $i = 1, 2$.
Therefore,
\[
  \mu_{\varepsilon,j}(V)
  \le \sum_{k=1}^\infty
  \mu_{\varepsilon,j}(\overline{B_{r_k}}(x_k))
  \le c \sum_{k=1}^\infty
  \mathcal{H}^\nu(\overline{B_{r_k}}(x_k))
  \le c \mathcal{H}^\nu(U)
  < c (\mathcal{H}^\nu(V) + \delta).
\]
Since $\delta > 0$ is arbitrary and
$\mu_{\varepsilon,j}$ is increasing monotonically as $\varepsilon$
decreases the statement of Step~1 follows.

\textbf{Step 2}.
Since $\varphi$ is Lipschitz continuous,
$\mathcal{J}(x, \mathrm{Pr} \circ \varphi)$
is bounded a.\,e. Let
$\|  \mathcal{J} \|_\infty = M$.
Denote
\[
  A_{kn} = \Bigl\{ x \in A :
  M \frac{k}{2^n} \le \mathcal{J}(x, \mathrm{Pr} \circ \varphi) < M \frac{k+1}{2^n}
  \Bigr\},
  \qquad
  k = 0, \ldots, 2^n.
\]
The sequence of simple functions
\[
  g_n(x) = \sum_{k=0}^{2^n} M\frac{k+1}{2^n} \chi_{A_{kn}}(x)
\]
is monotonically decreasing and converges to
$\mathcal{J}(x, \mathrm{Pr}_j \circ \varphi)$
for a.\,e.\ $x \in A$.
By step~1
\[
  \mu_j(A)
  \le \sum_{k=0}^{2^n} \mu_j(A_{kn})
  \le \sum_{k=0}^{2^n} M\frac{k+1}{2^n}
  \mathcal{H}^\nu(A_{kn})
  = \int\limits_A g_n(x) \, d\mathcal{H}^\nu(x).
\]
As $n \to \infty$
we obtain the statement of Theorem.
\end{proof}

The equality in Theorem~\ref{ThCoareaIneq} may be proved for
some classes of functions and for functions on certain Carnot
groups.

\begin{definition}
The a.\,e. differentiable function $f : E \subset \mathbb{G}$,
$E \subset \mathbb{G}$, has \emph{finite codistortion}
if
$\bigl| \mathop{\mathrm{adj}} \widehat{D}f(x) \bigr| = 0$
for a.\ e.\ points $x \in E$ such that
$\det \widehat{D} f(x) = 0$.
\end{definition}

\begin{corollary}
\label{ThCoareaFinCod}
Let $E \subseteq \mathbb{G}$,
$\varphi : E \to \mathbb{G}$ be Lipschitz continuous and
have finite codistortion.
Then for every measurable $A \subseteq E$
\[
  \intop_{\Pi_j} \mathcal{H}^1
  \big( (\mathrm{Pr}_j \circ \varphi)^{-1}(p) \cap A
  \big) \, d\mathcal{H}^{\nu-1}(p)
  = \frac{\Theta_{\Pi_j}}{\Theta_\mathbb{G}}
  \int\limits_A
  \bigl| \mathop{\mathrm{adj}} \widehat{D}\varphi(x) \langle X_j \rangle \bigr|
  \, d\mathcal{H}^\nu(x).
\]
\end{corollary}

\begin{proof}
Let $\Sigma \subset A$ be a set of points which are not
density points of $A$ or in which $\varphi$ is not
differentiable. Then $\Sigma$ is a null set and by
Lemma~\ref{LemmaCoareaAC}
\[
  \intop_{\Pi_j} \mathcal{H}^1
  \big( (\mathrm{Pr}_j \circ \varphi)^{-1}(p)
  \cap \Sigma
  \big) \, d\mathcal{H}^{\nu-1}(p) = 0.
\]
Next, let
$Z = \{ x \in A \setminus \Sigma : \det \widehat{D} \varphi(x) = 0 \}$. Then $\mathcal{J}(x, \mathrm{Pr} \circ \varphi) = 0$
on $Z$ and by Theorem~\ref{ThCoareaIneq}
\[
  \intop_{\Pi_j} \mathcal{H}^1
  \big( (\mathrm{Pr}_j \circ \varphi)^{-1}(p)
  \cap Z
  \big) \, d\mathcal{H}^{\nu-1}(p) = 0.
\]
The set $A \setminus (\Sigma \cup Z)$ is a set of density
points where $\det \widehat{D} \varphi(x) \ne 0$. It may be
split into a countable disjoint family of sets
$F_k$, $k \in \mathbb{N}$, such that
$\varphi|_{F_k}$ is bi-Lipschitz
(see, e.\,g. \cite[Remark~15]{VodEvs2014}).
Let $\psi : \varphi(F_k) \to F_k$ be the inverse map of
$\varphi|_{F_k}$. For every $p \in \Pi_j$ denote
\[
  I_{p,j} =
  \{ t \in \mathbb{R} : p \exp(t X_j) \in \varphi(F_k) \}.
\]
Then by Fubini theorem~\ref{ThFubini}
\[
  \intop_{\varphi(F_k)}
  \frac{\Theta_{\Pi_j}}{\Theta_\mathbb{G}}
  \left| X_j \psi(y) \right| \, d\mathcal{H}^\nu(y)
  =
  \intop_{\Pi_j} d\mathcal{H}^{\nu-1}(p)
  \intop_{I_{p,j}} \left| X_j \psi(p \exp(t X_j)) \right| \, dt.
\]
The function
$\gamma : t \mapsto \psi(p \exp(t X_j))$
is Lipschitz continuous and
$\gamma' = X_j \psi$, so
\begin{multline*}
  \intop_{I_{p,j}}
  \left| X_j \psi(p \exp(t X_j)) \right| \, dt
  =
  \intop_{I_{p,j}}
  \left| \gamma'(p \exp(t X_j)) \right| \, dt
  =
  \mathcal{H}^1 ( \gamma(I_p) ) \\
  =
  \mathcal{H}^1 \bigl(
  \psi(\mathrm{Pr}_j^{-1}(p) \cap \varphi(F_k) ) \bigr)
  =
  \mathcal{H}^1 \bigl(
  (\mathrm{Pr}_j \circ \varphi)^{-1}(p) \cap F_k \bigr).
\end{multline*}
Consequently,
\[
  \intop_{\varphi(F_k)}
  \frac{\Theta_{\Pi_j}}{\Theta_\mathbb{G}}
  \left| X_j \psi(y) \right| \, d\mathcal{H}^\nu(y)
  =
  \intop_{\Pi_j} \mathcal{H}^1 \bigl(
  (\mathrm{Pr}_j \circ \varphi)^{-1}(p) \cap F_k \bigr)
  \, d\mathcal{H}^{\nu-1}(p).
\]
Applying a change of variable formula for Lipschitz
functions~\cite[Theorem~7]{VodUkh1996} to the left hand side
we derive
\[
  \intop_{\varphi(F_k)}
  \left| X_j \psi(y) \right| \, d\mathcal{H}^\nu(y)
  = \intop_{F_k}
  \left| X_j \psi(\varphi(x)) \right|
  \bigl| \det \widehat{D}\varphi(x) \bigr|
  \, d\mathcal{H}^\nu(x)
  = \intop_{F_k}
  \bigl| \mathop{\mathrm{adj}} \widehat{D}\varphi(x)
  \langle X_j \rangle \bigr| \, d\mathcal{H}^\nu(x).
\]
Thus, the coarea equality is proved on $F_k$. Since $E$
is a disjoint union of $\Sigma$, $Z$ and $F_k$,
$k \in \mathbb{N}$, the corollary is proved.
\end{proof}

The structure of Carnot group places certain restrictions
on contact functions. Some Carnot groups are such that every
(regular enough) contact mapping has finite codistortion.
Let us show that is the case for the simplest non-Abelian
Carnot group, the Heisenberg group $\mathbb{H}^1$.

Recall, that $\mathbb{H}^1 = (\mathbb{R}^3, \cdot)$ is
a two-step Carnot group with the group operation
\[
  (x, y, z) \cdot (x', y', z')
  = \Bigl( x+x', y + y', z + z' + \frac{xy' - yx'}{2} \Bigr).
\]
The basis of its Lie algebra $\mathfrak{h}_1$
of left-invariant vector fields is
\[
  X = \partial_{x} - \frac{y}{2} \partial_z,
  \quad
  Y = \partial_{y} + \frac{x}{2} \partial_z,
  \quad
  Z = [X, Y] = \partial_{z}.
\]
The horizontal subspace is $H \mathbb{H}^1 =
\mathop{\mathrm{span}} \{ X, Y \}$.

\begin{lemma}
Every contact mapping on $\mathbb{H}^1$
has finite codistortion.
\end{lemma}

\begin{proof}
Let $\varphi = (f, g, h)$ be a mapping from (a subset of)
$\mathbb{H}^1$ to $\mathbb{H}^1$. And $p \in E$ be a point of
its $\mathcal{P}$-differentiability. For brevity we omit the
point. From contactness it follows
\[
  X\varphi = X f \, X + Xg \, Y,
  \qquad
  Y\varphi = Y f \, X + Yg \, Y.
\]
Since $\widehat{D}\varphi$ is a graded homomorphism of
Lie algebras,
\[
  \widehat{D}\varphi \langle Z \rangle
  = \widehat{D}\varphi \langle [X, Y] \rangle
  = [ \widehat{D}\varphi \langle  X \rangle,
    \widehat{D}\varphi \langle Y \rangle ]
  = [ X \varphi, Y \varphi ]
  = (Xf \, Yg - Xg \, Y f)Z.
\]
Thus, in the basis $X, Y, Z$ we have
\[
  \widehat{D} \varphi =
  \begin{pmatrix}
  Xf & Yf & 0 \\
  Xg & Yg & 0 \\
  0  &  0 & J_h \varphi
  \end{pmatrix},
\]
where $J_h \varphi = Xf \, Yg - Xg \, Y f$.
Then $\det \widehat{D}\varphi(p) = J_h \varphi(p)^2$.
Computing the adjoint matrix we obtain
\[
  \mathop{\mathrm{adj}} \widehat{D} \varphi
  =
  \begin{pmatrix}
  \phantom{+}Y g \, J_h \varphi & - Y f \, J_h \varphi & 0 \\
  -X g \, J_h \varphi & \phantom{+}X f \, J_h \varphi & 0 \\
  0 & 0 & J_h \varphi
  \end{pmatrix}
  .
\]
Consequently,
$\mathop{\mathrm{adj}} \widehat{D}\varphi(p) = 0$
as soon as $J_h \varphi(p) = 0$.
\end{proof}

One can check that the same is true for any Carnot
group $\mathbb{G}$ with $\dim H \mathbb{G} = 2$,
for instance, Engel group, and for direct products of such
groups.

\begin{corollary}
\label{ThCoareaH1}
The coarea equality~\ref{ThCoareaFinCod} holds for
any Lipschitz mapping $\varphi : E \to \mathbb{H}^1$,
$E \subset \mathbb{H}^1$.
\end{corollary}

\bibliographystyle{abbrv}
\bibliography{projection-coarea}

\noindent
Sergey Basalaev
(ORCID 0000-0002-4161-9948), \\
Novosibirsk State University, \\
1 Pirogova st., 630090 Novosibirsk Russia \\
\texttt{s.basalaev@g.nsu.ru}

\end{document}